\documentclass{amsart}
\usepackage{amssymb}

\newtheorem{theorem}{Theorem}[section]

\newtheorem{proposition}[theorem]{Proposition}

\renewcommand{\wr}{\mathop{\mathrm{wr}}}
\newcommand{\GL}{\mathop{\mathrm{GL}}}
\newcommand{\AGL}{\mathop{\mathrm{AGL}}}

\newcommand{\Aut}{\mathop{\mathrm{Aut}}}

\def\cent#1#2{{\bf C}_{#1}(#2)}
\def\norm#1#2{{\bf N}_{#1}(#2)}

\title[Regular abelian Carter subgroups]{Abelian Carter subgroups in finite permutation groups}

\author[E. Jabara]{Enrico Jabara}
\address{Enrico Jabara, Dipartimento di Filosofia e Beni Culturali, \newline
University  C\'a Foscari,  Dorsoduro 3484/D – 30123 Venezia, I-30100 Venezia, Italy. }
\email{jabara@unive.it}
\author[P. Spiga]{Pablo Spiga}
\address{Pablo Spiga, Dipartimento di Matematica Pura e Applicata,\newline
 University  Milano-Bicocca, Via Cozzi 53, 20126 Milano, Italy.}
\email{pablo.spiga@unimib.it}

\thanks{Address correspondence to Pablo Spiga: pablo.spiga@unimib.it.}



\begin{document}

\begin{abstract}
We show that  a finite permutation group containing a regular abelian self-normalizing subgroup is soluble.
\end{abstract}

\maketitle

\section{Introduction}\label{intro}

Let $G$ and $A$ be finite groups with $A$ acting on $G$ as a group of automorphisms.    When $\cent G A = 1$, it is customary to say that $A$ acts fixed-point-freely on $G$,
that is, $1$ is the only element of $G$ invariant by every element of $A$.
From the seminal work of Thompson~\cite{Th}, fixed-point-free groups of automorphisms have attracted considerable interest and have shown to be remarkably important within finite group theory. It is well-established that in many cases the condition $\cent G A=1$ forces the group $G$ to be soluble. One of the main contributions is  the result of Rowley~\cite{Ro}, showing that groups admitting a fixed-point-free automorphism are soluble. This result was then generalized by Belyaev and Hartley~\cite{BH}, showing that if a nilpotent group acts fixed-point-freely on $G$, then $G$ is soluble. In this paper,  as an application of this remarkable result, we prove the following.

\begin{theorem}\label{thm}
If $G$ is a finite transitive permutation group with a regular abelian self-normalizing subgroup, then $G$ is soluble.
\end{theorem}
(A self-normalizing nilpotent subgroup of a finite group is called a \textit{Carter} subgroup.)  Theorem~\ref{thm} is provoked by some recent investigations~\cite{DPS} on Cayley graphs over abelian groups. In fact, Dobson, Verret and the second author have recently proved a conjecture of Babai and Godsil~\cite[Conjecture~$2.1$]{BG} concerning the enumeration of Cayley graphs over abelian groups. In the approach in~\cite{DPS}, at a critical juncture~\cite[proof of Theorem~$1.5$]{DPS}, it is necessary to have structural information on finite permutation groups admitting a regular abelian Carter subgroup. In view of~\cite{DPS}, any  improvement in our understanding of the structure of the automorphism group of a Cayley graph over an abelian group requires a detailed description of the finite permutation groups containing a regular abelian Carter subgroup. Theorem~\ref{thm} is a first step in this direction.


The hypothesis of $A$ being abelian is rather crucial in Theorem~\ref{thm}. For instance, for each Mersenne prime $p=2^\ell-1$, the group $\mathrm{PSL}_2(p)$ acts primitively on the points of the projective line and contains a regular Carter subgroup isomorphic to the dihedral group of order $2^\ell$.

A well-known theorem of Carter~\cite{Carter} shows that every finite soluble group contains exactly one conjugacy class of Carter subgroups. On the other hand, a finite non-soluble group may have no Carter subgroup, as witnessed by the  alternating group $\Aut(5)$. However, using the Classification of the Finite Simple Groups, Vdovin~\cite{Vdovin} has extended the result of Carter by showing that, if a finite group $G$ contains a Carter subgroup, then every two distinct Carter subgroups of $G$ are conjugate. 

The main ingredients in our proof of Theorem~\ref{thm} are two results of fundamental importance in finite group theory. The first is the classification of  Li~\cite{chli} of the finite primitive groups containing an abelian regular subgroup. This classification will allow us to prove Theorem~\ref{thm} for primitive groups by a simple direct inspection (see Proposition~\ref{prim}). The second is the remarkable theorem of  Vdovin~\cite{Vdovin} showing that any two Carter subgroups of a finite group are conjugate. This result will allow us to use induction for proving Theorem~\ref{thm}. In particular, as~\cite{BH,chli,Vdovin} depend upon the Classification of the Finite Simple Groups, so does Theorem~\ref{thm}.

\section{Proof of Theorem~\ref{thm}}\label{proofs}

Before proving Theorem~\ref{thm} we need an auxiliary result.
\begin{proposition}\label{prim}Let $G$ be a finite primitive permutation group with a regular abelian subgroup $A$. Then either   $\norm G A\neq A$, or $G=A$ has prime order.
\end{proposition}
\begin{proof} 
The finite primitive groups containing an abelian regular subgroup are classified by  Li  in~\cite{chli}. In fact, from~\cite[Theorem~$1.1$]{chli}, we see that either the socle $N$ of $G$ is abelian, or $N$ is  the direct product of $\ell\geq 1$ pairwise isomorphic non-abelian simple groups and $G$ is endowed of the primitive product action. We consider these two cases separately.

Assume that $N$ is abelian. Then $N$ is an elementary abelian $p$-group, for some prime $p$. Since $N$ and $A$ both act regularly, we have $|N|=|A|$ and hence $A$ is a $p$-group. Therefore $AN$ is also a $p$-group. Suppose  $AN>A$. Then (from basic properties of $p$-groups) we have $\norm{AN}A>A $ and hence $\norm{G}A\neq A$. Suppose $AN=A$. Then $A=N\lhd G$ and, if $G>A$, then $\norm G A=G\neq A$, and if $G=A$, then $A$ must have prime order.

Assume that $N$ is non-abelian. Let $T_1,\ldots,T_\ell$ be the simple direct factors of $N$. So $N=T_1\times \cdots\times T_\ell$. From~\cite[Theorem~$1.1$~(2)]{chli}, we see that $G$ contains a normal subgroup $\widetilde{N}=\widetilde{T}_1\times \cdots \times \widetilde{T}_\ell$ with $\widetilde{T}_i$ having socle $T_i$ for each $i\in \{1,\ldots,\ell\}$, with $A\leq \widetilde{N}$ and with $$A=(A\cap \widetilde{T}_1)\times \cdots\times (A\cap \widetilde{T}_\ell).$$
Furthermore, for $i\in \{1,\ldots,\ell\}$, each of the possible pairs $(\widetilde{T}_i,A\cap \widetilde{T}_i)$ is listed in~\cite[Theorem~$1.1$~(2)~(i),\ldots,(iv)]{chli}. An easy inspection on each of these four cases reveals that $\norm{{\widetilde{T}_i}}{{A\cap \widetilde{T}_i}}>(A\cap \widetilde{T}_i)$, for each $i\in \{1,\ldots,\ell\}$. Therefore $$\norm G A\geq \norm {\widetilde{N}}A=\norm{\widetilde{T}_1}{A\cap \widetilde{T}_1}\times \cdots \times \norm{\widetilde{T}_\ell}{A\cap \widetilde{T}_\ell}>(A\cap\widetilde{T}_1)\times \cdots\times (A\cap\widetilde{T}_\ell)=A.$$ 
\end{proof}

\begin{proof}[Proof of Theorem~$\ref{thm}$]
We argue by contradiction and among all possible counterexamples we choose $G$ and $A$ with $|G|+|A|$ as small as possible. We let $\Omega$ be the transitive set acted upon by $G$.

As $G$ is insoluble and $A$ is self-normalizing, from Proposition~\ref{prim} we deduce that  $G$ is imprimitive. Among all non-trivial systems of imprimitivity for $G$ choose  $\mathcal{B}$ with blocks of minimal size. We denote by $G_{(\mathcal{B})}$ the pointwise stabilizer of $\mathcal{B}$, that is, $G_{(\mathcal{B})}=\{g\in G\mid B^g=B, \,\textrm{ for each }B\in \mathcal{B}\}$. In particular, $G_{(\mathcal{B})}$ is the kernel of the action of $G$ on $\mathcal{B}$. Also, for $B\in \mathcal{B}$, we write $G_{\{B\}}=\{g\in G\mid B^g=B\}$ for the setwise stabilizer of $B$ and $G_{(B)}=\{g\in G\mid \alpha^g=\alpha, \,\textrm{for each }\,\alpha\in B\}$ for the pointwise stabilizer of $B$. Moreover, given a subgroup $X$ of $G$, we let $X^{\mathcal{B}}$ denote the permutation group induced by $X$ on $\mathcal{B}$. In particular, $X^{\mathcal{B}}\cong XG_{(\mathcal{B})}/G_{(\mathcal{B})}$.

Observe that  $A^{\mathcal{B}}$ is a transitive abelian subgroup of $G^{\mathcal{B}}$, and hence $A^{\mathcal{B}}$ acts regularly on $\mathcal{B}$. We show that $A^{\mathcal{B}}$ is a Carter subgroup of $G^{\mathcal{B}}$. As $A^{\mathcal{B}}\cong AG_{(\mathcal{B})}/G_{\mathcal{B}}$, it suffices to show that $\norm G {AG_{(\mathcal{B})}}=AG_{(\mathcal{B})}$. Let $g\in \norm {G}{AG_{(\mathcal{B})}}$.  Now $(AG_{(\mathcal{B})})^g=AG_{(\mathcal{B})}$ and hence $A,A^g\leq AG_{(\mathcal{B})}$. As $A$ and $A^g$ are Carter subgroups of $G$, they are also Carter subgroups of $AG_{(\mathcal{B})}$. In particular, by~\cite{Vdovin}, $A$ and $A^g$ are conjugate in $AG_{(\mathcal{B})}$. Therefore there exists $h\in G_{(\mathcal{B})}$ with $A^{h}=A^g$. Thus $gh^{-1}\in \norm G A=A$ and $g\in AG_{(\mathcal{B})}$. 

Since $A^{\mathcal{B}}$ is a Carter subgroup of $G^{\mathcal{B}}$, we see that $G^{\mathcal{B}}$ and $A^{\mathcal{B}}$ satisfy the hypothesis of Theorem~\ref{thm}.
As $\mathcal{B}$ is a non-trivial system of imprimitivity, $|A^{\mathcal{B}}|<|A|$ and $|G^{\mathcal{B}}|\leq |G|$ and hence, by minimality, $G^{\mathcal{B}}\cong G/G_{(\mathcal{B})}$ is soluble. In particular, $G_{(\mathcal{B})}$ is insoluble.

Suppose that $AG_{(\mathcal{B})}<G$. Then $A$ is a regular abelian Carter subgroup of $AG_{(\mathcal{B})}$ with $|AG_{(\mathcal{B})}|<|G|$. By minimality, $AG_{(\mathcal{B})}$ is soluble and hence so is $G_{(\mathcal{B})}$, a contradiction. Thus
\begin{equation}\label{eq1}
G=AG_{(\mathcal{B})}.
\end{equation}

From~\eqref{eq1}, we have $G^{\mathcal{B}}=A^{\mathcal{B}}$ and hence $G$ acts regularly on $\mathcal{B}$. In particular, $G_\omega\leq G_{(\mathcal{B})}$, for every $\omega\in \Omega$, and $G_{\{B\}}=G_{(\mathcal{B})}$, for every $B\in \mathcal{B}$. Moreover, by the minimality of $|B|$, the group $G_{(\mathcal{B})}=G_{\{B\}}$ acts primitively on $B$.

We show that
\begin{equation}\label{eq2-}
\norm {G}{A_0}=\cent {G}{A_0},\qquad \textrm{for every }A_0\leq A.
\end{equation}
Clearly the right hand side is contained in the left hand side. Let $g\in \norm {G}{A_0}$. Now $A_0$ is centralized by $A$ and by $A^g$, and so by $\langle A,A^g\rangle$. As $A$ and $A^g$ are Carter subgroups of $\langle A,A^g\rangle$, by~\cite{Vdovin} there exists $x\in \langle A,A^g\rangle$ with $A^g=A^x$. Thus $gx^{-1}\in \norm G A=A$. In particular, $gx^{-1}$ and $x$ centralize $A_0$, and hence so does $g$. Thus $g\in \cent {G}{A_0}$.

Applying~\eqref{eq2-} with $A_0=A\cap G_{(\mathcal{B})}$ we obtain 
\begin{equation}\label{eq2}
\norm {G}{A\cap G_{(\mathcal{B})}}=\cent {G}{A\cap G_{(\mathcal{B})}}.
\end{equation}

Fix $B_1\in \mathcal{B}$ and write $H=G_{\{B_1\}}^{B_1}$, that is, the permutation group induced by $G_{\{B_1\}}$ on the block $B_1$. By the Embedding Theorem~\cite[Theorem~$1.2.6$]{meldrum}, the group $G$ is permutation isomorphic to a subgroup of the wreath product $W=H\wr A^{\mathcal{B}}$ (in its natural imprimitive product action on the cartesian product $B_1\times\mathcal{B}$). Write $\ell=|\mathcal{B}|$. In particular, under this embedding $G_{(\mathcal{B})}$ is identified with a subgroup of $H^\ell$. As $G_{(\mathcal{B})}$ is insoluble and primitive in its action on $B_1$, it follows that
\begin{equation}\label{eq3}
H\quad\textrm{is an insoluble primitive permutation group.}
\end{equation}


Let $N$ be a minimal normal subgroup of $G$ with $N\leq G_{(\mathcal{B})}$. We now divide the proof in two cases: depending on whether $N$ is abelian or non-abelian.

\smallskip

\noindent\textsc{Case I: }$N$ is non-abelian.

\smallskip

\noindent Observe that $A$ is a regular abelian Carter subgroup of $AN$ and that $AN$ is not soluble, hence by minimality  $G=AN$. The modular law gives $G_{(\mathcal{B})}=(A\cap G_{(\mathcal{B})})N$ and thus $H=G_{(\mathcal{B})}^{B_1}=(A\cap G_{(\mathcal{B})})^{B_1}N^{B_1}$. It is easy to check that $(A\cap G_{(\mathcal{B})})^{B_1}$ is a regular abelian subgroup of $H$ and that $N^{B_1}$ is a characteristically simple normal subgroup of $H$. 

We now appeal to the classification of Li~\cite{chli} of the primitive groups with a regular abelian subgroup. Using the notation in~\cite{chli}, we see that in all the groups $X$ in~\cite[Theorem~$1.1$(2)]{chli} the regular abelian subgroup $G$ of $X$ acts trivially by conjugation on the simple direct factors of the socle of $X$. In particular, as $(A\cap G_{(\mathcal{B})})^{B_1}$ acts transitively on the simple direct factors of $N^{B_1}$, we deduce that $N^{B_1}$ is a non-abelian simple group, that is, $H$ is an almost simple group. Furthermore, 
$(H,(A\cap G_{(\mathcal{B})})^{B_1})$ is one of the pairs (denoted by $(\widetilde{T}_i,G_i)$) in~\cite[Theorem~$1.1$(2)~(i), \ldots, (iv)]{chli}.

Let $T_1,\ldots,T_{\ell'}$ be the simple direct factors of $N$ with $T_1$ acting faithfully on $B_1$ and with $N_{(B_1)}=T_2\times\cdots \times T_{\ell'}$. Observe that $T_1,\ldots,T_{\ell'}$ are pairwise isomorphic and that $A$ acts transitively by conjugation on $\{T_1,\ldots,T_{\ell'}\}$. Let $\pi_1:N\to T_1$ be the projection on the first coordinate and set $A_1=\pi_1(N \cap A)$. In particular, replacing $G$ by a suitable conjugate in $\Aut(N)$, we may suppose that $N\cap A=\{(a,\ldots,a)\mid a\in A_1\}\leq A_1^{\ell'}$. Now, a direct inspection on each of the pairs in~\cite[Theorem~$1.1$(2)~(i), \ldots, (iv)]{chli} shows that there exists $h\in \norm {T_1}{A_1}\setminus \cent {T_1}{A_1}$. The element $g=(h,\ldots,h)\in T_1\times \cdots\times T_{\ell'}=N$ and $g\in \norm G {A\cap N}\setminus \cent G {A\cap N}$, contradicting~\eqref{eq2-} (applied with $A_0=A\cap N$). 

\smallskip

\noindent\textsc{Case II: }$N$ is an elementary abelian $p$-group for some prime $p$.

\smallskip

\noindent Now $N^{B_1}$ is a normal elementary abelian $p$-subgroup of $G_{(\mathcal{B})}^{B_1}=H$. As $H$ is primitive,  $N^{B_1}$ is the socle of $H$ and $H$ is a primitive group of affine type, that is, $H$ is isomorphic to a primitive subgroup of the affine general linear group $\AGL_d(p)$ for some  $d\geq 1$.

Let $V$ be the socle of $H$, let $V^\ell$ be the socle of $H^\ell$ (as a subgroup of $W$) and let $U=V^\ell\cap G$.  Let $\pi_1:G_{(\mathcal{B})}\to H$ be the projection onto the first coordinate and set $A_1=\pi_1(A\cap G_{(\mathcal{B})})$. Since $A_1=(A\cap G_{\mathcal{B}})^{B_1}$, we see that $A_1$ is a regular abelian subgroup of $H$. Moreover, as $A$ acts transitively by conjugation on the $\ell$ coordinates of $H^\ell$, replacing $G$ by a suitable conjugate, under the embedding $G\leq W$ we have 
\begin{equation}\label{eq4}
A\cap G_{(\mathcal{B})}=\{(a,\ldots,a)\mid a \in A_1\}.
\end{equation}

Suppose that $A_1=V$.  Then $A\cap G_{(\mathcal{B})}\leq V^\ell\cap G=U$. Moreover, $G_{(\mathcal{B})}/U$ is isomorphic to a subgroup of the direct product $\GL_d(p)^\ell$.  Since $G_{(\mathcal{B})}$ is insoluble, so is $X=G_{(\mathcal{B})}/U$. Furthermore, $A$ acts as a group of automorphisms on $X$. From~\cite[Theorem~0.11]{BH}, we have $\cent X A\neq 1$ and hence there exists $g\in G_{(\mathcal{B})}\setminus U$ with $[g,A]\leq U$. As $A$ acts transitively on the elements of $\mathcal{B}$,  $g=(v_1h,v_2h,\ldots,v_\ell h)$ for some $h\in H$ and some $v_1,\ldots,v_\ell\in V$. As $g\notin U$, we have $h\notin V$ (if $h\in V$, then $g\in V^\ell\cap G=U$, a contradiction). As $V=\cent H V$, we get  $h\notin \cent H V$ and $h\in H=\norm H V$. From~\eqref{eq4} it is clear that $g\in \norm G{A\cap G_{(\mathcal{B})}}\setminus \cent G{A\cap G_{(\mathcal{B})}}$, which contradicts~\eqref{eq2}.

Suppose that $A_1\neq V$. We show that there exists $u\in \norm U{A\cap G_{(\mathcal{B})}}\setminus\cent U{A\cap G_{(\mathcal{B})}}$, from which the theorem will immediately follow by contradicting~\eqref{eq2}. Since $A_1 V>A_1$ and $\cent H{A_1}={A_1}$, there exists $\bar{v}\in \norm V {A_1}\setminus \cent H{A_1}$. We label the elements of $\mathcal{B}$ by $B_1,B_2,\ldots,B_{\ell}$, where the action of $g=(g_1,\ldots,g_\ell)\in G_{(\mathcal{B})}$ on $B_i$ is given by the element  $g_{i}$ on the $i^{\mathrm{th}}$ coordinate of $g$.

Assume that $G_{(B_1)}=1$. Let $u\in U$ with $u^{B_1}=\bar{v}$, that is, $u=(\bar{v},v_2,\ldots,v_\ell)$, for some $v_2,\ldots,v_\ell\in V$. By construction $u$ normalizes $A\cap G_{(\mathcal{B})}$ modulo $G_{(B_1)}$ and hence $u$ normalizes $A\cap G_{(\mathcal{B})}$ because $G_{(B_1)}=1$. 

Assume that $G_{(B_1)}\neq 1$. Let $B\in \mathcal{B}\setminus\{B_1\}$ with $(G_{(B_1)})^{B}\neq 1$. Relabelling the set $\mathcal{B}$ if necessary, we may assume that $B=B_2$. As $G_{(B_1)}\lhd G_{(\mathcal{B})}$ and $G_{(\mathcal{B})}$ acts primitively on $B_2$, we get $V\leq (G_{(B_1)})^{B_2}$. We show that $$U=U_{(B_1)}U_{(B_2)}.$$ Let $z=(v_1,v_2,v_3,\ldots,v_\ell)\in U$. As  
$V=(U_{(B_1)})^{B_2}$, there exists $x=(1,v_2,v_3',\ldots,v_\ell')\in U_{(B_1)}$, for some $v_3',\ldots,v_\ell'\in V$. Clearly, $zx^{-1}=(v_1,1,v_3v_3'^{-1},\ldots,v_\ell v_\ell'^{-1})\in U_{(B_2)}$.

Now we show that  $$(\bar{v},\bar{v},v_3,\ldots,v_\ell)\in U,\qquad \textrm{for some }v_3,\ldots,v_\ell\in V.$$ As $U=U_{(B_1)}U_{(B_2)}$, we get $V=(U_{(B_1)})^{B_2}$ and hence there exists $y=(1,\bar{v},v_3',\ldots,v_\ell')\in U_{(B_1)}$. Similarly,  as $V=(U_{(B_2)})^{B_1}$, there exists $x=(\bar{v},1,v_3'',\ldots,v_\ell'')\in U_{(B_2)}$. Now, $xy\in U$ and the first two coordinates of $xy$ are $\bar{v},\bar{v}$. 

Assume that $G_{(B_1)}\cap G_{(B_2)}=1$, that is, $G_{(B_1\cup B_2)}=1$. Let $u=(\bar{v},\bar{v},v_3,\ldots,v_\ell)\in U$. Now, from~\eqref{eq4}, we see that $u$ normalizes $A\cap G_{(\mathcal{B})}$ modulo $G_{(B_1\cup B_2)}$ and hence $u$ normalizes $A\cap G_{(\mathcal{B})}$ because $G_{(B_1\cup B_2)}=1$.

Assume that $G_{(B_1)}\cap G_{(B_2)}\neq 1$.  Let $B\in \mathcal{B}\setminus\{B_1,B_2\}$ with $(G_{(B_1\cup B_2)})^{B}\neq 1$. Relabelling the set $\mathcal{B}$ if necessary, we may assume that $B=B_3$. As $G_{(B_1\cup B_2)}\lhd G_{(\mathcal{B})}$ and $G_{(\mathcal{B})}$ acts primitively on $B_3$, we get $V\leq (G_{(B_1\cup B_2)})^{B_3}$. We show that $$U=U_{(B_1\cup B_2)}U_{(B_3)}.$$ Let $z=(v_1,v_2,v_3,\ldots,v_\ell)\in U$. As  
$V=(U_{(B_1\cup B_2)})^{B_3}$, there exists $x=(1,1,v_3,v_4',\ldots,v_\ell')\in U_{(B_1\cup B_2)}$, for some $v_4',\ldots,v_\ell'\in V$. Clearly, $zx^{-1}\in U_{(B_3)}$.

Now we show that  $$(\bar{v},\bar{v},\bar{v},v_4,\ldots,v_\ell)\in U,\qquad \textrm{for some }v_4,\ldots,v_\ell\in V.$$ Indeed, from above there exists $z=(\bar{v},\bar{v},w_3,\ldots,w_\ell)\in U$, for some $w_3,\ldots,w_\ell\in V$. Moreover,
as $U=U_{(B_1\cup B_2)}U_{(B_3)}$, we get $V=U^{B_3}=(U_{(B_1\cup B_2)})^{B_3}$ and hence there exists $y=(1,1,w_3^{-1}v,w_4',\ldots,w_\ell')\in U_{(B_1\cup B_2)}$. Now, $zy\in U$ and the first three coordinates of $zy$ are $\bar{v},\bar{v},\bar{v}$.

Assume that $G_{(B_1\cup B_2)}\cap G_{(B_3)}=1$, that is, $G_{(B_1\cup B_2\cup B_3)}=1$. Let $u=(\bar{v},\bar{v},\bar{v},v_4,\ldots,v_\ell)\in U$. Now, from~\eqref{eq4}, we see that $u$ normalizes $A\cap G_{(\mathcal{B})}$ modulo $G_{(B_1\cup B_2\cup B_3)}$ and hence $u$ normalizes $A\cap G_{(\mathcal{B})}$ because $G_{(B_1\cup B_2\cup B_3)}=1$.

Finally, the case $G_{(B_1\cup B_2\cup B_3)}\neq 1$ follows easily by induction applying the same two-step argument as in the previous paragraphs.
\end{proof}

\end{document}